\documentclass[12pt,reqno]{amsart} 
\usepackage{amssymb,amscd,url}
\usepackage{enumerate}
\usepackage[all,cmtip]{xy}

\begin{document}

\allowdisplaybreaks


\title[Families of elliptic curves over $\PP^n$]{Heights and the specialization map for Families of elliptic curves over $\PP^n$}
\date{\today}
\author[W.P. Wong]{Wei Pin Wong}
\email{wongpin101@math.brown.edu}
\address{Mathematics Department, Box 1917
         Brown University, Providence, RI 02912 USA}


\hyphenation{ca-non-i-cal semi-abel-ian}

\newtheorem{theorem}{Theorem}
\newtheorem{lemma}[theorem]{Lemma}
\newtheorem{conjecture}[theorem]{Conjecture}
\newtheorem{proposition}[theorem]{Proposition}
\newtheorem{corollary}[theorem]{Corollary}
\newtheorem*{claim}{Claim}
\newtheorem*{theoremB}{Theorem B}
\newtheorem*{theoremA}{Theorem A}
\newtheorem*{corollaryB}{Corollary B}
\theoremstyle{definition}
\newtheorem*{definition}{Definition}
\newtheorem{example}[theorem]{Example}
\newtheorem*{remark}{Remark}
\newtheorem{question}[theorem]{Question}

\theoremstyle{remark}
\newtheorem*{acknowledgement}{Acknowledgements}


\newenvironment{notation}[0]{%
  \begin{list}%
    {}%
    {\setlength{\itemindent}{0pt}
     \setlength{\labelwidth}{4\parindent}
     \setlength{\labelsep}{\parindent}
     \setlength{\leftmargin}{5\parindent}
     \setlength{\itemsep}{0pt}
     }%
   }%
  {\end{list}}

\newenvironment{parts}[0]{%
  \begin{list}{}%
    {\setlength{\itemindent}{0pt}
     \setlength{\labelwidth}{1.5\parindent}
     \setlength{\labelsep}{.5\parindent}
     \setlength{\leftmargin}{2\parindent}
     \setlength{\itemsep}{0pt}
     }%
   }%
  {\end{list}}
\newcommand{\Part}[1]{\item[\upshape#1]}

\renewcommand{\a}{\alpha}
\renewcommand{\b}{\beta}
\newcommand{\g}{\gamma}
\renewcommand{\d}{\delta}
\newcommand{\e}{\epsilon}
\newcommand{\f}{\varphi}
\newcommand{\bfphi}{{\boldsymbol{\f}}}
\renewcommand{\l}{\lambda}
\renewcommand{\k}{\kappa}
\newcommand{\lhat}{\hat\lambda}
\newcommand{\m}{\mu}
\newcommand{\bfmu}{{\boldsymbol{\mu}}}
\renewcommand{\o}{\omega}
\renewcommand{\r}{\rho}
\newcommand{\rbar}{{\bar\rho}}
\newcommand{\s}{\sigma}
\newcommand{\sbar}{{\bar\sigma}}
\renewcommand{\t}{\tau}
\newcommand{\z}{\zeta}

\newcommand{\D}{\Delta}
\newcommand{\G}{\Gamma}
\newcommand{\F}{\Phi}
\renewcommand{\L}{\Lambda}

\newcommand{\ga}{{\mathfrak{a}}}
\newcommand{\gb}{{\mathfrak{b}}}
\newcommand{\gn}{{\mathfrak{n}}}
\newcommand{\gp}{{\mathfrak{p}}}
\newcommand{\gP}{{\mathfrak{P}}}
\newcommand{\gq}{{\mathfrak{q}}}

\newcommand{\Abar}{{\bar A}}
\newcommand{\Ebar}{{\bar E}}
\newcommand{\kbar}{{\bar k}}
\newcommand{\Kbar}{{\bar K}}
\newcommand{\Pbar}{{\bar P}}
\newcommand{\Sbar}{{\bar S}}
\newcommand{\Tbar}{{\bar T}}

\newcommand{\Acal}{{\mathcal A}}
\newcommand{\Bcal}{{\mathcal B}}
\newcommand{\Ccal}{{\mathcal C}}
\newcommand{\Dcal}{{\mathcal D}}
\newcommand{\Ecal}{{\mathcal E}}
\newcommand{\Fcal}{{\mathcal F}}
\newcommand{\Gcal}{{\mathcal G}}
\newcommand{\Hcal}{{\mathcal H}}
\newcommand{\Ical}{{\mathcal I}}
\newcommand{\Jcal}{{\mathcal J}}
\newcommand{\Kcal}{{\mathcal K}}
\newcommand{\Lcal}{{\mathcal L}}
\newcommand{\Mcal}{{\mathcal M}}
\newcommand{\Ncal}{{\mathcal N}}
\newcommand{\Ocal}{{\mathcal O}}
\newcommand{\Pcal}{{\mathcal P}}
\newcommand{\Qcal}{{\mathcal Q}}
\newcommand{\Rcal}{{\mathcal R}}
\newcommand{\Scal}{{\mathcal S}}
\newcommand{\Tcal}{{\mathcal T}}
\newcommand{\Ucal}{{\mathcal U}}
\newcommand{\Vcal}{{\mathcal V}}
\newcommand{\Wcal}{{\mathcal W}}
\newcommand{\Xcal}{{\mathcal X}}
\newcommand{\Ycal}{{\mathcal Y}}
\newcommand{\Zcal}{{\mathcal Z}}

\renewcommand{\AA}{\mathbb{A}}
\newcommand{\BB}{\mathbb{B}}
\newcommand{\CC}{\mathbb{C}}
\newcommand{\FF}{\mathbb{F}}
\newcommand{\GG}{\mathbb{G}}
\newcommand{\NN}{\mathbb{N}}
\newcommand{\PP}{\mathbb{P}}
\newcommand{\QQ}{\mathbb{Q}}
\newcommand{\RR}{\mathbb{R}}
\newcommand{\ZZ}{\mathbb{Z}}
\newcommand{\CQ}{\overline{\mathbb{Q}}}
\newcommand{\bfa}{{\mathbf a}}
\newcommand{\bfb}{{\mathbf b}}
\newcommand{\bfc}{{\mathbf c}}
\newcommand{\bfd}{{\mathbf d}}
\newcommand{\bfe}{{\mathbf e}}
\newcommand{\bff}{{\mathbf f}}
\newcommand{\bfg}{{\mathbf g}}
\newcommand{\bfp}{{\mathbf p}}
\newcommand{\bfr}{{\mathbf r}}
\newcommand{\bfs}{{\mathbf s}}
\newcommand{\bft}{{\mathbf t}}
\newcommand{\bfu}{{\mathbf u}}
\newcommand{\bfv}{{\mathbf v}}
\newcommand{\bfw}{{\mathbf w}}
\newcommand{\bfx}{{\mathbf x}}
\newcommand{\bfy}{{\mathbf y}}
\newcommand{\bfz}{{\mathbf z}}
\newcommand{\bfA}{{\mathbf A}}
\newcommand{\bfF}{{\mathbf F}}
\newcommand{\bfB}{{\mathbf B}}
\newcommand{\bfD}{{\mathbf D}}
\newcommand{\bfG}{{\mathbf G}}
\newcommand{\bfI}{{\mathbf I}}
\newcommand{\bfM}{{\mathbf M}}
\newcommand{\bfzero}{{\boldsymbol{0}}}

\newcommand{\Aut}{\operatorname{Aut}}
\newcommand{\codim}{\operatorname{codim}}
\newcommand{\cufr}{\operatorname{cufr}}
\newcommand{\diag}{\operatorname{diag}}
\newcommand{\Disc}{\operatorname{Disc}}
\newcommand{\Div}{\operatorname{Div}}
\newcommand{\Dom}{\operatorname{Dom}}
\newcommand{\End}{\operatorname{End}}
\newcommand{\Fbar}{{\bar{F}}}
\newcommand{\Gal}{\operatorname{Gal}}
\newcommand{\GL}{\operatorname{GL}}
\newcommand{\Hom}{\operatorname{Hom}}
\newcommand{\Index}{\operatorname{Index}}
\newcommand{\Image}{\operatorname{Image}}
\newcommand{\hhat}{{\hat h}}
\newcommand{\Ker}{{\operatorname{ker}}}
\newcommand{\lcm}{\operatorname{lcm}}
\newcommand{\Lift}{\operatorname{Lift}}
\newcommand{\Mat}{\operatorname{Mat}}
\newcommand{\Mor}{\operatorname{Mor}}
\newcommand{\Moduli}{\mathcal{M}}
\newcommand{\Norm}{{\operatorname{\mathsf{N}}}}
\newcommand{\notdivide}{\nmid}
\newcommand{\normalsubgroup}{\triangleleft}
\newcommand{\NS}{\operatorname{NS}}
\newcommand{\onto}{\twoheadrightarrow}
\newcommand{\ord}{\operatorname{ord}}
\newcommand{\Orbit}{\mathcal{O}}
\newcommand{\Per}{\operatorname{Per}}
\newcommand{\Perp}{\operatorname{Perp}}
\newcommand{\PrePer}{\operatorname{PrePer}}
\newcommand{\PGL}{\operatorname{PGL}}
\newcommand{\Pic}{\operatorname{Pic}}
\newcommand{\Prob}{\operatorname{Prob}}
\newcommand{\Pfr}{\operatorname{fr}}
\newcommand{\Qbar}{{\bar{\QQ}}}
\newcommand{\rank}{\operatorname{rank}}
\newcommand{\Rat}{\operatorname{Rat}}
\newcommand{\Resultant}{\operatorname{Res}}
\renewcommand{\setminus}{\smallsetminus}
\newcommand{\sgn}{\operatorname{sgn}} 
\newcommand{\SL}{\operatorname{SL}}
\newcommand{\Span}{\operatorname{Span}}
\newcommand{\Spec}{\operatorname{Spec}}
\newcommand{\Support}{\operatorname{Supp}}
\newcommand{\sq}{\operatorname{sq}}
\newcommand{\sqfr}{\operatorname{sqfr}}
\newcommand{\tors}{{\textup{tors}}}
\newcommand{\Trace}{\operatorname{Trace}}
\newcommand{\tr}{{\textup{tr}}} 
\newcommand{\UHP}{{\mathfrak{h}}}    
\newcommand{\ve}{\mathbf T} 
\newcommand{\<}{\langle}
\renewcommand{\>}{\rangle}
\newcommand{\longhookrightarrow}{\lhook\joinrel\longrightarrow}
\newcommand{\longonto}{\relbar\joinrel\twoheadrightarrow}
\newcommand{\mapsfrom}{\mathrel{\reflectbox{\ensuremath{\longmapsto}}}}

\begin{abstract}
For $n\geq 2$, let $K=\overline{\QQ}(\PP^n)=\overline{\QQ}(T_1, \ldots, T_n)$. Let $E/K$ be the elliptic curve defined by a minimal Weierstrass equation $y^2=x^3+Ax+B$, with $A,B \in \overline{\QQ}[T_1, \ldots, T_n]$. There's a canonical height $\hat{h}_{E}$ on $E(K)$ induced by the divisor $(O)$, where $O$ is the zero element of $E(K)$. On the other hand, for each smooth hypersurface $\G$ in $\PP^n$ such that the reduction mod $\G$  of $E$, $E_{\G} / \overline{\QQ}(\G)$ is an elliptic curve with the zero element $O_\G$, there is also a canonical height $\hat{h}_{E_{\G}}$ on $E_{\G}(\overline{\QQ}(\G))$ that is induced by $ (O_\G)$. We prove that for any $P \in E(K)$, the equality $\hat{h}_{E_{\G}}(P_\G)/ \deg \G =\hat{h}_{E}(P)$ holds for almost all hypersurfaces in $\PP^n$. As a consequence, we show that for infinitely many $t \in \PP^n(\overline{\QQ})$, the specialization map $\sigma_t : E(K) \rightarrow E_t(\overline{\QQ})$ is injective.
\end{abstract}
\maketitle
\section{Introduction}
For $n \geq 2$, let $K=\overline{\QQ}(T_1, \ldots, T_n)$ be the function field of the projective space $\PP^n$ over $\overline{\QQ}$. Let $E/K$ be the elliptic curve over $K$ defined by the Weierstrass equation  
$$ Y^2Z=X^3+AXZ^2+BZ^3 $$ with $A,B \in \CQ[T_1, \ldots, T_n]$ such that it is minimal with respect to all but the infinity prime divisor in $\PP^n$, i.e. $$0\leq\ord_D(A) <4 \qquad \text{or}\qquad  0\leq \ord_D(B)<6  $$  for every prime divisor $D$ that is not the infinity hyperplane. 
For any hypersurface $\G$ in $\PP^n$ 
such that the reduction mod $\G$ of $E/K$, $E_\G$ is an elliptic curve over $\CQ(\G)$, we have a group homomorphism 
\begin{align*}
E(K) &\longrightarrow E_\G(\CQ(\G)) \\
P &\longmapsto P_\G.
\end{align*} 
Let $O$  be the zero element of $E(K)$ and $\hat{h}_E$ be the canonical height on $E(K)$ corresponded to the divisor $(O)$. Similarly, we denote $O_\G$ to be the zero element of $E_\G(\CQ(\G))$ and $\hat{h}_{E_\G}$ to be the canonical height on $E_\G(\CQ(\G))$ corresponded to the divisor $(O_\G)$. As a partial generalization of Silverman's theorem (\cite{SHS}, Theorem B) in the case of elliptic curves, we prove the following theorem that relates these heights:

\begin{theoremA} Given any $P \in E(K)$, there exists a set $B_P$ consisting of a finite number of codimension-two subvarieties in $\PP^n$, such that if $\G \subset \PP^n$ is a smooth hypersurface that does not contain any subvariety in $B_P$ and $E_\G/\CQ(\G)$ is non-singular, then
$$\frac{\hat{h}_{E_\G}(P_\G)}{\deg \G}= \hat{h}_E(P). $$
\end{theoremA}

One can view $E/K$ as the generic fiber of an abelian scheme $\pi : \Ecal \longrightarrow U$, for some Zariski open dense subset $U \subset \PP^n$. Then for all $t \in U(\CQ)$, the specialization map
\begin{align*}
\s_t : E(K) &\longrightarrow E_t(\CQ) \\
P &\longmapsto P_t
\end{align*} is a group homomorphism. In the setting of general abelian varieties over $K=k(\PP^n)$ for any number field $k$, N\'eron (\cite{NPA}) proved that there are infinitely many $t \in \PP^n(k)$ such that $\s_t$ is injective. N\'eron proved this result by showing that the set of  $t \in \PP^n(k)$ for which  $\s_t$ is not injective is thin. Later, Masser (\cite{MSF}) proved a stronger result in a more general setting of abelian varieties $A$ defined over $K=k(V)$, the function field of a projective variety $V$ defined over a number field $k$. Masser proved that the specialization map $\s_t$ is ``almost always'' injective by showing that the set of $t \in V(k)$ with Weil height bounded by $h$ and such that $\s_t$ is not injective lies on a hypersurface of degree bounded by a constant power of $h$. Moreover, in the case where $V=C$ is a smooth curve, Silverman (\cite{SHS}, Theorem C) proved that the set of $t \in C(\overline{k})$ for which $\s_t$ is not injective is a set of bounded height. For the special case of elliptic curves over $K=\CQ(\PP^n)$, we obtain a similar result by combining Silverman's theorem and Theorem A:

\begin{theoremB} With notations as explained above and let $d \geq 1$ be a fixed integer, then there exist infinitely many smooth curves $C/\CQ$ of degree $d$ in $\PP^n$ that do not lie in the complement of $U$ and such that the set 
$$\{ t \in C(\CQ) \cap U(\CQ) \ | \ \s_t \text{ is not injective }\} $$
is a set of bounded height. Furthermore, the union of all such curves is Zariski dense in $\PP^n$.
\end{theoremB} 

When $K=\CQ(C)$ is the function field of a smooth curve $C/\CQ$, the hypersurfaces in $C$ are $\CQ$-points $t$ on the curve $C$. There is a Weil height $h_{C,\phi}$ (that depends on the closed embedding of $\phi : C \longhookrightarrow \PP^m$) on $C(\CQ)$ and we normalize it to $h_C:=h_{C,\phi}/\deg \phi$. Silverman (\cite{SHS}, Theorem B) proved that for $\pi : A\longrightarrow C$, a family of abelian varieties over a smooth curve, we have
$$\lim_{\substack{t \in  C(\CQ) \\ h_{C}(t) \rightarrow \infty} }   \frac{\hat{h}_{A_{t}}(P_{t})}{h_{C}(t)} = \hat{h}_{A_\eta}(P_\eta), $$ where $A_\eta/\CQ(C)$ is the generic fiber of $A$. Our Theorem A is analogous to the elliptic surface version of Silverman's theorem despite the fact that there's no limit involved and we view $\deg \G$ as the Weil height of $\G$. The main reason for this difference is the type of global field over which the reduction  $E_\G$ is defined: over a curve $C$, $\G=\{t\}$ is a point and $E_t$ is an elliptic curve over a number field, whereas over $\PP^n$ for $n \geq 2$, $E_{\G}$ is an elliptic curve over a function field. Consequently, the canonical height $\hat{h}_{E_{t}}$ on $E_{t}(\CQ)$ is derived from a Weil height in a number field and the canonical height $\hat{h}_{E_{\G}}$ on $E_{\G}(\CQ(\G))$ is derived from  a Weil height in a function field. As we shall see in section \ref{lemma}, the theory of Weil heights over the function field of a smooth hypersurface is related to the intersection theory of divisors in the projective space, i.e. B\'ezout's theorem and this is where the degree of $\G$ comes into play.

On the other hand, instead of looking at the fibers over smooth hypersurfaces of $\PP^n$, the author (\cite{WWP}) also tried to generalize Silverman's theorem 
by looking at the fibers over points of $\PP^n$ of $E/\QQ(\PP^n)$. As remarked in that paper, the limit of the quotient $\hat{h}_{E_{t}}(P_t)/h_{\PP^n}(t) $ doesn't exist when $h_{\PP^n}(t)$ tends to infinity for $t \in \PP^n(\CQ)$. Thus, the author studied instead the average value of the quotient $\hat{h}_{E_{t}}(P_t)/h_{\PP^n}(t)$ over rational $t \in \PP^n(\QQ)$ with bounded height. As the bound tends to infinity, the author showed that this average remains finite and uniformly bounded below by a nonzero constant for all non-torsion $P$ in $E(\QQ(\PP^n))$. Even in this simple case of $\PP^n$ with $n\geq 2$, it is still an open question whether this average converges, let alone whether it converges to $\hat{h}_{E}(P)$. 

In section \ref{HF}, we remind the readers of the definition of the Weil height and canonical height on an elliptic curve defined over $k(V)$, the function field of an arbitrary smooth projective variety $V$ defined over a number field $k$. From now onwards, we will use the convention that varieties and subvarieties are always irreducible. Next, in section \ref{lemma}, we will prove some lemmas that relate the Weil heights on $E(K)$ and $E_\G(\CQ(\G))$. In fact, the first few lemmas show that the equality in Theorem A holds for Weil heights on $E_\G$, for a hypersurface $\G$ that does not contain certain codimension-two subvarieties in $\PP^n$ that depend on $P$. To pass from Weil heights to canonical heights, we will replace $P$ by $[2^m]P$, divide the Weil heights by $3 \cdot 2^{2m}$ and take the limit $m \rightarrow \infty$. The rest of the lemmas are used to show that the dependence of $\G$ on infinitely many $[2^m]P$ can be reduced to just not containing a finite number of codimension-two subvarieties. In the last section, we will apply Theorem A to prove Theorem B.

\section{Heights in Function Fields} \label{HF}
Let $V$ be a smooth projective variety defined over a number field $k$. We fix a closed embedding $\phi : V \longhookrightarrow \PP^n$ and thus the degree map $\deg_\phi$ is well-defined on $\Div_{\bar{k}}(V)$, i.e. for any prime divisor $D \in \Div_{\bar{k}}(V)$, $\deg_\phi (D)$ is the degree of the projective subvariety $\phi(D)$ in $\PP^n$ and we extend it linearly to all divisors.  For any $P=[f_0: \ldots: f_r]$ in $\PP^r(\bar{k}(V))$, the Weil height of $P$ is given by
$$h_{\PP^r(\bar{k}(V))}(P):= \sum_{\G} \max_i\{-\ord_{\G}(f_i)\}\deg_\phi(\G), $$ where the summation is taken over all prime divisors in $\Div_{\bar{k}}(V)$. This Weil height has the following geometrical interpretation:

\begin{proposition}(Lang \cite{LFD} Chapter 3, Proposition 3.2 ) \label{LH} Let $V$ be a projective variety in $\PP^n$, non-singular in codimension one, defined over a number field $k$, and let $P$ be a point in the projective space $\PP^r$, rational over $k(V)$. Let $f_P:V \dashrightarrow \PP^r$ be the rational map defined over $k$, determined by $P$. Then
$$h_{\PP^r(\bar{k}(V))}(P)=\deg f_P^{-1}(L) $$for any hyperplane $L$ of $\PP^r$, such that $ f_P^{-1}(L)$ is defined, the degree being that in the given projective embedding of $V$ in $\PP^n$.
\end{proposition}

Let $E/k(V)$ be an elliptic curve over the function field $k(V)$ defined by a Weierstrass equation and $O$ be the zero element. Then this embedding of $E$ into $\PP^2_{k(V)}$ is induced by the very ample divisor $3(O)$. Let $h_E$ be the Weil height corresponds to this embedding, i.e. for any $P=[x:y:z] \in E(k(V))$, 
$$h_E(P)=h_{\PP^2(k(V))}([x:y:z]). $$  Since the divisor $3(O)$ is even, the canonical height $\hat{h}_E$ induced by $(O)$ is defined by
$$\hat{h}_{E}(P):=\frac{1}{3}\lim_{n \rightarrow \infty}\frac{h_{E}([n]P)}{n^2}. $$ 
This canonical height is a quadratic function on $E(\overline{k(V)})$ with some nice properties. Readers are invited to consult Lang \cite{LFD} for more details on height functions (Chapter 3, 4) and canonical heights on abelian varieties (Chapter 5).

\section{Lemmas}  \label{lemma}
Let $n \geq 2$ and $S_0, \ldots, S_n$ be the homogeneous coordinate functions of the projective space $\PP^n$ over $\CQ$. Let $K=\CQ(\PP^n)$ be the function field of $\PP^n$. If we denote $T_1:=\frac{S_1}{S_0}, \ldots, T_n:=\frac{S_n}{S_0}$, then $K=\CQ(T_1, \ldots, T_n)$ and the infinity hyperplane $H_\infty$ is the plane defined by $S_0=0$ . Let $E/K$ be an elliptic curve defined by the Weierstrass equation 
\begin{equation}\label{WE} Y^2Z=X^3+AXZ^2+BZ^3 \end{equation} with $A,B \in \CQ[T_1, \ldots, T_n]$ such that it is minimal with respect to all prime divisors in $\PP^n$ that are not $H_\infty$, i.e. $$0\leq \ord_D(A) <4 \qquad \text{or}\qquad  0\leq \ord_D(B)<6  $$  for all prime divisors $D \neq H_\infty$. Such a Weierstrass equation can always be obtained via change of variables. The discriminant 
$$\D_E=-16(4A^3+27B^2) $$ is a non-zero element in $\CQ[T_1, \ldots, T_n]$. 

For any smooth hypersurface (smooth irreducible codimension-one subvariety) $\G$ in $\PP^n$, the local ring at $\G$ is a discrete valuation ring $\Ocal_\G$ with maximal ideal $\frak{m}_\G$. We denote the reduction map $\Ocal_\G \longrightarrow \Ocal_\G/\frak{m}_\G= \CQ(\G)$ by $f \longmapsto f_\G$. Then for $\G \neq H_\infty$, the reduction of $E$ modulo $\G$, $E_\G$ is the cubic curve (possibly singular) over $\CQ(\G)$ defined by 
$$Y^2Z=X^3+A_\G XZ^2+B_\G Z^3,$$ as $A,B \in \Ocal_\G$ for all $\G \neq H_\infty$. Let 
$$E_\G(\CQ(\G))_{\text{ns}}:=\{ \text{ non-singular points of }E_\G(\CQ(\G))\} $$ and
$$E_0(K)=\{ P \in E(K) \ | \ P_\G \in E_\G(\CQ(\G))_{\text{ns}}\}, $$
then we have a group homomorphism (\cite{SAE} Chapter VII, Proposition 2.1. Note: the condition that $K$ is complete in Proposition 2.1 is only needed to prove the surjectivity of the reduction map, which is not needed here.)
\begin{align*}
E_0(K) &\longrightarrow E_\G(\CQ(\G))_{\text{ns}} \\
P=[x:y:z] &\longmapsto P_\G=[x_\G: y_\G: z_\G],
\end{align*} where at least one of $x,y,z \in \Ocal_\G$ is not in $\frak{m}_\G$. All the lemmas in this section are based on the setting mentioned above.

As we have seen from the previous section, the Weil height $h_E(P)$ of a point $P\in E(\overline{k}(V))$ is directly related to the rational map induced by $P$. So in order to study the relationship between $h_{E}(P)$ and $h_{E_\G}(P_\G)$, we look at the corresponding induced rational maps. In general, if we have a rational map $f : V \dashrightarrow W$ between projective varieties and $V$ is non-singular, the indeterminacy locus of $f$, which we denote as $I_f$, has codimension at least $2$ (\cite{SAT} Chapter III, Proposition 3.5b) and so we have the following isomorphism (\cite{HAG} Chapter II, Proposition 6.5b) of divisor class groups induced by the inclusion $j:U:= V\backslash I_f \longhookrightarrow V$:
\begin{align*}
j^*: \Pic(V) &\stackrel{\cong}{\longrightarrow} \Pic(U)\\
c &\longmapsto  c \cap U
\end{align*}with projective closure in $V$ as the inverse homomorphism. With this, the pull-back $f^{-1}$ of divisor class is well-defined in the following sense:

\begin{definition} Given a rational map $f : V \dashrightarrow W$ between projective varieties and $V$ is non-singular. Denote $I_f$ to be the indeterminacy locus of $f$ and $U:= V\backslash I_f$. Let $j: U \longhookrightarrow V$ be the inclusion map and $f|_{U}:U\longrightarrow W$ be the restriction of $f$ on $U$, which is a morphism between quasi-projective varieties. Then $f|_{U}^* : \Pic(W)  \longrightarrow \Pic(U)$ is well-defined and we define $f^{-1}:=(j^*)^{-1} \circ f|_{U}^* : \Pic(W) \longrightarrow \Pic(V)$, i.e. $f^{-1}$ is defined such that the following diagram commutes:
$$\xymatrix{
&\Pic(U)  
&\Pic(W) \ar@{->}[l]_{f|_{U}^*} \ar@{->}[dl]^{f^{-1}} \\
&\Pic(V)  \ar@{->}[u]^{j^*}_\cong 
}$$
\end{definition}

The following lemma serves as the first step in relating the Weil heights $h_{E}(P)$ and $h_{E_\G}(P_\G)$.

\begin{lemma}\label{Le1} For $n\geq 2$, let $f : \PP^n \dashrightarrow \PP^r$ be a rational map and $\G \stackrel{\iota}{\longhookrightarrow} \PP^n$ be a smooth hypersurface, where all the varieties are defined over an algebraically closed field $k$. Suppose $\G$ does not contain any codimension-two component of the indeterminacy locus of $f$, then 
$$(f \circ \iota)^{-1}(c)=\iota^*(f^{-1}(c)), $$ for any divisor class $c \in \Pic(\PP^r)$.
\end{lemma}
\begin{proof} Since $f$ is a rational map between smooth projective varieties, the indeterminacy locus $I_f$ has codimension at least $2$, i.e. each irreducible component of $I_f$ has dimension at most $n-2$.

Define the quasi-projective varieties $U^o:=\PP^n\backslash I_f$ and $\G^o:=\G\backslash I_f \cap \G$ and we have the following commutative diagram 

$$\xymatrix{
\G^o \ar@{^{(}->}[d]_{j_{\G^o}} \ar@{^{(}->}[r]^{\iota_{\G^o}}
&U^o \ar@{->}[r]^{f|_{U^o}} \ar@{^{(}->}[d]_{j_{U^o}}
&\PP^r \\
\G \ar@{^{(}->}[r]^{\iota_\G} 
&\PP^n \ar@{-->}[ur]_f}
$$
where all the $\iota$ and $j$ are inclusion maps and $f|_{U^o}$ is the restriction of $f$ on $U^o$. Since $\G$ does not contain any codimension-two component of $I_f$, the codimension of $I_f \cap \G$ in $\G$ is at least $2$ also. Thus, we have the following isomorphisms of divisor class groups induced by $j_{\G^o}$ and $j_{U^o}$:
\begin{align*} j^*_{U^o}: \Pic(\PP^n) &\stackrel{\cong}{\longrightarrow} \Pic(U^o)  &j^*_{\G^o}: \Pic(\G) &\stackrel{\cong}{\longrightarrow} \Pic(\G^o)\\
 c &\longmapsto c \cap U^o & D &\longmapsto D \cap \G^o,
\end{align*} with projective closure in the corresponding varieties as the inverse maps. With these isomorphisms and the above commutative diagram, we obtain the following commutative diagram of divisor class groups:

$$\xymatrix{
\Pic(\G^o)
&\Pic(U^o)  \ar@{->}[l]_{\iota_{\G^o}^*} 
&\Pic(\PP^r) \ar@{->}[l]_{f|_{U^o}^*} \ar@{->}[dl]^{f^{-1}} \\
\Pic(\G) \ar@{->}[u]^{j_{\G^o}^*}_\cong 
&\Pic(\PP^n)  \ar@{->}[u]^{j_{U^o}^*}_\cong \ar@{->}[l]_{\iota_\G^*} 
}$$

Pull-backs are functorial on the category of varieties (\cite{FIT} Chapter 2, Section 2.2, \cite{HAG} Appendix A, Theorem 1.1), so for any divisor class $c \in \Pic(\PP^r)$, we have 
$$\iota_{\G^o}^*(f|_{U^o}^*(c))=(f|_{U^o} \circ \iota_{\G^o})^*(c). $$ Then by taking $(j_{\G^o}^*)^{-1}$, the projective closure in $\G$ and tracing the commutative diagram, we get
\begin{align*}
\iota_{\G}^*(f^{-1}(c))&=(j_{\G^o}^*)^{-1}\left(\iota_{\G^o}^*(f|_{U^o}^*(c))\right) \\& = (j_{\G^o}^*)^{-1}\left((f|_{U^o} \circ \iota_{\G^o})^*(c)\right) \\ &= (f \circ \iota_\G)^{-1}(c), \end{align*} which is the statement of the lemma.
\end{proof}

\begin{lemma} \label{leh}Let $f_P : \PP^n \dashrightarrow \PP^2$ be the rational map induced by $P \in E(K)$ with indeterminacy locus $I_{f_P}$. If $\G \stackrel{\iota}{\longhookrightarrow} \PP^n$ is a smooth hypersurface such that $E_\G$ is non-singular and $\G$ does not contain any codimension-two component of $I_{f_P}$, then 
\begin{equation}\frac{h_{E_\G}(P_\G)}{\deg \G}= h_E(P). \label{E1}\end{equation}
\end{lemma}
\begin{proof} Let $g_{P_\G} : \G \dashrightarrow \PP^2$ be the rational map induced by $P_\G=[x_\G: y_\G: z_\G]$. Then $g_{P_\G}$ factors through
$$\xymatrix{
\G \ar@{-->}[r]^{g_{P_\G}} \ar@{^{(}->}[d]_{\iota_{\G}} 
&\PP^2  \\
\PP^n \ar@{-->}[ur]_{f_P} 
}$$
Let $L \in \Pic(\PP^2)$ be the divisor class of hyperplanes, then we have 
\begin{align*}
h_{E_\G}(P_\G)&= \deg g_{P_\G}^{-1}(L) & (\text{by Proposition \ref{LH}})\\
&=\deg (f_P \circ \iota_\G)^{-1}(L)\\
&= \deg \iota_\G^* (f_P^{-1}(L)) &(\text{by Lemma \ref{Le1}}).
\end{align*}
But the divisor class $\iota_\G^* (f_P^{-1}(L))$ is the divisor class of the intersection of $\G$ and $f_P^{-1}(L)$ (\cite{FIT} Chapter 8, Example 8.1.7 or \cite{HAG} Chapter II, Exercises 6.2). Then by B\'ezout's theorem, we get $\deg \iota_\G^* (f_P^{-1}(L)) = \deg \G \cdot \deg f_P^{-1}(L)$, which gives 
$$h_{E_\G}(P_\G)=\deg \G \cdot \deg f_P^{-1}(L)=\deg \G \cdot h_E(P) $$ by Proposition \ref{LH} again.
\end{proof}

In order to obtain Theorem A from Lemma \ref{leh}, we will replace $P$ in equation (\ref{E1}) by $[2^m]P$, then divide the resulting equation by $3 \cdot 2^{2m}$ and take the limit $m\rightarrow \infty$. This leads us to study the indeterminacy locus of $f_{[2^m]P}$ when $m$ goes to infinity. 

\begin{lemma} \label{DL} For $t \in \PP^n(\CQ)\backslash H_\infty$ such that the specialized point $P_t:=f_P(t)=[x(t):y(t):z(t)]$ is a non-singular point on the specialized (possibly singular) cubic curve $E_t$ defined by \begin{equation} \label{WEt}Y^2Z=X^3+A(t)XZ^2+B(t)Z^3, \end{equation} $([2]P)_t=f_{[2]P}(t)$ is also well-defined and $([2]P)_t=[2]P_t$.
\end{lemma}   
\begin{proof} For $E/K$ defined by the Weierstrass equation $(\ref{WE})$ with $P=[x:y:z] \in E(K)$ where $x,y,z \in \CQ[S_0,\ldots,S_n]$, the doubling formula (\cite{SAT} Chapter IV, page 324) is
\begin{align}&[2]P=\notag \\*
&[2yz\left(\left(3x^2+Az^2\right)^2-2xz(2y)^2\right): \notag \\
&-\left(3x^2+Az^2\right)^3+z(2y)^2\left(8x^3+2Axz^2-Bz^3-y^2z\right): \notag \\
&8y^3z^3]. \label{D1} 
\end{align}
By substituting the relation $x^3=y^2z-Axz^2-Bz^3$ into the first two coordinates and then factor away the common factor $z^2$, we also have
\begin{align}&[2]P=\notag \\
&[2y\left(xy^2-3Ax^2z-9Bxz^2+A^2z^3\right): \notag \\
&4y^2\left(7y^2-6Axz-9Bz^2\right)-\left(27(y^2-Axz-Bz^2)^2+27Ax^4+9A^2x^2z^2+A^3z^4\right): \notag \\
&8y^3z]. \label{D2} 
\end{align}
Since by assumption $t\in\PP^n(\CQ)\backslash H_\infty$ such that $P_t=f_P(t)=[x(t):y(t):z(t)]$ is defined and is a non-singular point on the cubic curve $E_t$, so $y(t)$ and $z(t)$ cannot be both zero. It is clear that $f_{[2]P}$ is defined when both $y(t)$ and $z(t)$ are not zero. When $y(t)=0$, then $3x(t)^2+A(t)z(t)^2$ is not zero or otherwise $P_t$ is a singular point on $E_t$. So formula $(\ref{D1})$ gives $([2]P)_t=[0:1:0]$, which is not a surprise as $P_t=[x(t):0:z(t)]$ is always a 2-torsion point with the given Weierstrass equation (\ref{WEt}). When $z(t)=0$, then $x(t)=0$ and $y(t) \neq 0$, since $P_t$ is a point on $E_t$. Formula $(\ref{D2})$ gives $([2]P)_t=[0:y(t)^4:0]=[0:1:0]$. Thus the induced rational map $f_{[2]P}$ is defined at such $t$. Lastly, formulae $(\ref{D1})$ and $(\ref{D2})$ are the same (with $A,B$ replaced by $A(t)$ and $B(t)$) doubling formula on the cubic curve $E_t$, so $([2]P)_t=[2]P_t$. 
\end{proof}

The next lemma allows us to reduce to the case where the points $t \in \PP^n(\CQ)$ for which $P_t$ is a singular point on $E_t$ are of codimension at least $2$ in $\PP^n(\CQ)$.

\begin{lemma} \label{LNS} With the minimality of the Weierstrass equation that defines $E/K$, for any $P \in E(K)$, there exists a natural number $N$ such that $([N]P)_\G$ is not a singular point on $E_\G(\CQ(\G))$ for all smooth hypersurfaces $\G \neq H_\infty$. 
\end{lemma}

\begin{proof} For any smooth hypersurface $\G$, its local ring $\Ocal_\G$ is a DVR with fraction field $K$ and residue field $\CQ(\G)$, which is a perfect field.  With the minimality assumption on the Weierstrass equation, $E/K$ is defined by a Weierstrass equation with coefficients in $\Ocal_\G$ and minimal with respect to the prime divisor $\G$. Thus, the smooth part of the Weierstrass model gives the N\'eron model of $E/K$ and the quotient group $E(K)/E_0(K)$ is finite (\cite{SAT} Chapter IV, Corollary 9.2d). So there exists $n_\G \in \NN$ such that $([n_\G]P)_\G$ is not a singular point on $E_\G(\CQ(\G))$. Since $E_\G/\CQ(\G)$ is singular if and only if $\G$ is a component of the discriminant locus and there are only finitely many such components $\G_1, \ldots, \G_\ell$. So $N:=\lcm\{n_{\G_1}, \dots, n_{\G_\ell}\}$ will satisfy the lemma.
\end{proof}

So far, the lemmas above allow us to control the indeterminacy of $f_{[2^m]P}$ over $\PP^n(\CQ)\backslash H_\infty$. Since $H_\infty$ is the pole of $A$ and $B$, the reduction mod $H_\infty$ of $E$ is not defined. To overcome this, we use the following change of variables in $K$ to obtain an elliptic curve $E'$ that is $K$-isomorphic to $E$, defined by a Weierstrass equation whose coefficients don't have pole at $H_\infty$:

Let $k:=\max\left\{ \left\lceil \frac{\deg A}{4}\right\rceil, \left\lceil \frac{\deg B}{6}\right\rceil \right\}$ and 
\begin{equation}\label{WE'} E'/K: \qquad  Y^2Z=X^3+\left(\frac{S_0}{S_1}\right)^{4k}AXZ^2+\left(\frac{S_0}{S_1}\right)^{6k}BZ^3,
\end{equation}
which is the elliptic curve obtained from $E/K$ by the group isomorphism 
\begin{align*} E/K &\stackrel{\approx}{\longrightarrow} E'/K \\
P=[x:y:z] &\longmapsto P'=[x':y':z']:=\left[\left(\frac{S_0}{S_1}\right)^{2k}x: \left(\frac{S_0}{S_1}\right)^{3k}y:z \right].
\end{align*} The choice of $k$ is to make sure the Weierstrass equation $(\ref{WE'})$ is still minimal with respect to all prime divisors except at the hyperplane $H'_\infty$ defined by $S_1=0$. With this notation, the next lemma relates the behavior of $f_P$ and $f_{P'}$ at $t \in H_\infty(\CQ)\backslash H'_\infty(\CQ)$.

\begin{lemma} \label{PL} Let $t=[0,1:*:\ldots:*] \in H_\infty(\CQ) \backslash H'_\infty(\CQ)$ such that the rational map $f_P$ is not defined at $t$. Then exactly one of the following statements is true:
\begin{itemize}
\item[a)] $f_{P'}$ is not defined at $t$,
\item[b)] $P'_t$ is a singular point on  $E'_t(\CQ)$,
\item[c)] $P'_t$ is a torsion point of order $2$ on $E'_t(\CQ)$.
\end{itemize}
\end{lemma}
\begin{proof} If we write $P=[x:y:z] \in E(K)$ where the homogeneous polynomials $x,y,z \in \CQ[S_0, \ldots, S_n]$ have no common irreducible factor, then the induced rational map $f_P$ is not defined at $t \in \PP^n(\CQ)$ if and only if $x(t)=y(t)=z(t)=0$, because $f_P$ maps from $\PP^n$ and so there's no way to alter it. The corresponding point $P'$ is $[S_0^{2k}S_1^kx: S_0^{3k}y: S_1^{3k}z]$. Since at least one of $x,y,z$ has no $S_0$ as its factor, we consider all possible cases:

Case 1: $z$ doesn't have $S_0$ as its factor. Then the only possible irreducible common factor of the homogeneous polynomials $S_0^{2k}S_1^kx$, $S_0^{3k}y$, $S_1^{3k}z$ is $S_1$ and even if we factor out this common factor, the induced rational map $f_{P'}$ is still not defined at $t=[0:1:*:\ldots:*]$.  

Case 2: $y$  doesn't have $S_0$ as its factor. In this case, the only possible common factor of the homogeneous polynomials $S_0^{2k}S_1^kx, S_0^{3k}y, S_1^{3k}z$ are $S_1$ and $S_0$. After removing the common factors, let $P'=[x'':y'':z'']$. Since $y(t)=0$ and the only possible factor we factor away from $y$ in $ S_0^{3k}y$ is $S_1$, thus $y''(t)=0$. Also, $t=[0:1:*:\ldots:*] \in H_\infty(\CQ) \backslash H'_\infty(\CQ)$ by assumption. If $P'_t$ is well-defined, then it is a point on the cubic curve $E_t'(\CQ)$ with zero $Y$-coordinate, which implies $P'_t$ is either a singular point or a torsion point of order $2$ on $E'_t(\CQ)$.

Case 3: $x$  doesn't have $S_0$ as its factor. The highest power of $S_0$ that can be a common factor of the homogeneous polynomials $S_0^{2k}S_1^kx, S_0^{3k}y, S_1^{3k}z$ is $S_0^{2k}$.  After removing the common factors, let $P'=[x'':y'':z'']$, then we see that $y''$ always has $S_0^k$ as a factor. Thus $y''(t)=0$ and the argument follows as Case 2. 
\end{proof}

\section{Proof of Theorem A} \label{P}
Now, we put ourselves in the setting described in Section \ref{lemma}. From Lemma \ref{leh}, if $\G$ is a smooth hypersurface that does not contain any codimension-two component of the indeterminacy locus of $f_{[2^m]P}$ for infinitely many $m$, then equality (\ref{E1}) holds for these $[2^m]P$ and thus dividing equation (\ref{E1}) by $3 \cdot 2^{2m}$ and then taking the  limit $m\rightarrow \infty$ over all $m$ such that equality (\ref{E1}) holds will give us the equality in Theorem A. A priori, such $\G$ might have to avoid containing infinitely many codimension-two subvarieties in $\PP^n$ corresponding to the infinitely many rational maps $f_{[2^m]P}$, we will prove that this is not a problem. In fact, we will prove that besides finitely many codimension-two subvarieties in $\PP^n$ that might be common components of the indeterminacy locus of all $f_{[2^m]P}$, $m \in \NN$, other  codimension-two components of the indeterminacy locus of a particular $f_{[2^m]P}$ will never reappear as components of the indeterminacy locus of  $f_{[2^\ell]P}$ for any $\ell > m$.

Since the canonical height is a quadratic form, we may replace $P$ by $[N]P$ in Theorem A. Given any $P \in E(K)$, we replace $P$ by $[N]P$ as in Lemma $\ref{LNS}$ to assure that $P_\G$ is not a singular point of $E_\G$ for all smooth hypersurfaces $\G \neq H_\infty$. So if we let $S_P$ be the set of all $t \in \PP^n(\CQ)\backslash H_\infty(\CQ)$ such that the specialized point $P_t$ is a singular point on the specialized fiber $E_t$, then $S_P$ is an algebraic set (since singularity is an algebraic property) of codimension at least two in $\PP^n\backslash H_\infty$. In addition, we note that $S_P$ is contained in the discriminant locus $\D_E=0$. Similarly, we may assume the same for the corresponding point $P'$ on the isomorphic curve $E'$ defined by equation $(\ref{WE'})$, i.e. $P'_\G$ is not a singular point of $E'_\G$  for all smooth hypersurfaces $\G \neq H'_\infty$ and $S_{P'}$ has codimension at least $2$ in $\PP^n(\CQ)\backslash H'_\infty(\CQ)$.

As before, let $I_{f_P}$ be the indeterminacy locus of the rational map $f_P$ induced by $P$. Then for all $t \in \PP^n(\CQ)\backslash H_\infty(\CQ)$ that are not in $I_{f_P} \cup S_P$, $f_P(t)=P_t$ is well-defined and is not a singular point on the cubic curve $E_t$. By lemma \ref{DL} and induction on $m$, $f_{[2^m]P}(t)=\left( [2^m]P\right)_t=[2^m]P_t$ is well-defined for all $m \in \NN$. Thus, $$I_{f_{[2^m]P}} \cap \left(\PP^n(\CQ)\backslash H_\infty(\CQ) \right) \subset I_{f_P}\cup S_P, \qquad\text{ for all }m \in \NN.$$ The exact same argument works for $P'$ and so we also have $$I_{f_{[2^m]P'}} \cap \left(\PP^n(\CQ)\backslash H'_\infty(\CQ) \right) \subset I_{f_{P'}}\cup S_{P'}\qquad\text{ for all } m \in \NN.$$

For $$t=[0:1:*: \dots:*] \in \left(I_{f_{[2^m]P}} \cap H_\infty(\CQ)\right) \backslash H'_\infty(\CQ),$$ if $t \not \in  I_{f_{P'}}\cup S_{P'}$, then lemma \ref{DL} says that $f_{[2^m]P'}(t)=([2^m]P')_t$ is defined and non-singular on $E'_t$ because $([2^m]P')_t=[2^m]P'_t$. In particular, lemma \ref{PL} implies that in fact $[2^m]P'$ is a point of order $2$ on $E'_t$. Thus, for $\ell > m$, $([2^\ell]P')_t=[2^\ell]P'_t=[2^{\ell-m}]\left([2^m]P'_t\right)=O_{E'_t}$, which implies that $t$ is not in $I_{f_{[2^\ell]P}}$ by the contrapositive statement of lemma \ref{PL}. Notice that we are left with $t \in I_{f_{[2^m]P}} \cap H_\infty(\CQ)\cap H'_\infty(\CQ)$, which is an algebraic set of codimension at least $2$ in $\PP^n(\CQ)$. To summarize, we have shown that given $P\in E(K)$, by replacing $P$ by $[N]P$ as in Lemma $\ref{LNS}$ if necessary, so that $P_\G$ and $P'_\G$ are not singular points of $E_\G$ and $E'_\G$ respectively for all smooth hypersurface $\G$, for any $m \in \NN$, if 
$$
  t \in I_{f_{[2^m]P}}
  \quad\text{and}\quad
  t \notin I_{f_P} \cup I_{f_{P'}} \cup S_P \cup S_{P'} \cup (H_\infty(\CQ)\cap H'_\infty(\CQ)),
$$
then for any $\ell>m$, $t \not \in I_{f_{[2^\ell]P}}$. In particular, if we let
$$B_P:=\left\{ H_\infty\cap H'_\infty, \ \text{dimension-$(n-2)$ components of }I_{f_P}\cup I_{f_{P'}}\cup S_P\cup S_{P'}   \right\}, $$then for any $m \in \NN$, if $C$ is a dimension-$(n-2)$ component of $I_{f_{[2^m]P}}$ that is not in $B_P$, then $C$ is not a component of  $I_{f_{[2^\ell]P}}$ for all $\ell > m$. 

To complete the proof for Theorem A, we consider any smooth hypersurface $\G \subset \PP^n$ that is not a component of the discriminant locus, so $E_\G$ is an elliptic curve over $\CQ(\G)$. Furthermore, if $\G$ does not contain any subvariety in $B_P$, then we have just shown that $\G$ does not contain any codimension-two component of $I_{f_{[2^m]P}}$ for all big enough $m \in \NN$. So by lemma \ref{leh}, we have $$\frac{h_{E_\G}([2^m]P_\G)}{\deg \G}=\frac{h_{E_\G}(\left([2^m]P\right)_\G)}{\deg \G}= h_E([2^m]P) $$ for all big enough $m \in \NN$. Dividing this equation by $3 \cdot 2^{2m}$ and then taking the limit $m\rightarrow \infty$ will give us Theorem A.

\section{Injectivity of the specialization map} 
We continue to use all the notation that we defined in section \ref{lemma}. $E/K$ is the generic fiber of an abelian scheme $\pi : \Ecal \longrightarrow U$, for some Zariski open dense subset $U \subset \PP^n$. Then for all $t \in U(\CQ)$, the specialization map
\begin{align*}
\s_t : E(K) &\longrightarrow E_t(\CQ) \\
P &\longmapsto P_t
\end{align*} is a group homomorphism (\cite{SLM} Chapter 11, page 152). The following theorem says that there are infinitely many smooth curves $C \subset \PP^n$ such that if we restrict our attention to those specialization maps $\s_t$ correspond to $t \in C(\CQ)$, then most of the $\s_t$ are injective off of set of bounded height:

\begin{theoremB} Given any integer $d \geq 1$, there exist infinitely many smooth curves $C/\CQ$ of degree $d$ in $\PP^n$ that do not lie in the complement of $U$ and such that the set 
$$\{ t \in C(\CQ) \cap U(\CQ) \ | \ \s_t \text{ is not injective }\} $$
is a set of bounded height. Furthermore, the union of all such curves is Zariski dense in $\PP^n$.
\end{theoremB} 

\begin{proof}
If the $K/\CQ$-trace of $E$ (\cite{LFD}, Chapter 6) is non-zero, then by the constraint of dimension, the $K/\CQ$-trace of $E$ is a dimensional one abelian variety, i.e. $E$ is $K$-isomorphic to $E_0 \times_{\CQ} K$ for an elliptic curve $E_0/\CQ$. So there is a Zariski dense open subset $U^0 \subset \PP^n$ such that for all $t \in U^0(\CQ)$, the fibers $E_t$ are isomorphic to $E_0$, thus the specialization maps for $\s_t: E_0 \times_{\CQ} K \longrightarrow E_t$  are injective for all $t \in U^0(\CQ)$. So in particular, any curve $C \subset \PP^n$ of degree $d$ that does not lie in the complement of $U^0$ will satisfies Theorem B and their union is $\PP^n$.

We now suppose that the $K/\CQ$-trace of $E$ is zero, which implies that $E(K)$ is finitely generated, a result due to Lang-N\'eron (\cite{LFD}, Chapter 6, Theorem 2). For any non-zero $P=[x:y:z] \in E(K)$, the algebraic set
$$ \{ t \in U(\CQ)\ | \ [x(t):y(t):z(t)]=\s_t(P)=O=[0:1:0] \}$$ has dimension at most $n-1$. Since the subgroup of torsion points $E(K)_{\text{tors}}$ is finite, so we may shrink $U$ by discarding finitely many subvarieties of dimension at most $n-1$ and assume that 
$$E(K)_{\text{tors}} \longrightarrow E_t(\CQ) $$
is injective for all $t \in  U(\CQ)$. On the other hand, We will use the canonical height pairing to prove that the specialization map is injective on the free part of $E(K)$. The canonical height pairing on $E(K)$ is 
$$\< \cdot  , \cdot \>_E : E(K) \times E(K) \longrightarrow \RR $$
$$\<P, Q\>_E \longmapsto \frac{1}{2}\left(\hat{h}_{E}(P+Q)-\hat{h}_{E}(P)-\hat{h}_{E}(Q) \right).$$
Likewise, for any smooth hypersurface $\G \subset \PP^n$ such that $E_\G$ is non-singular, the canonical height pairing on $E_{\G}(\CQ(\G))$ is
$$\< \cdot  , \cdot \>_{E_{\G}} : E_{\G}(\CQ(\G)) \times E_{\G}(\CQ(\G)) \longrightarrow \RR $$
$$\<P_{\G}, Q_{\G}\>_{E_{\G}} \longmapsto \frac{1}{2}\left(\hat{h}_{E_{\G}}(P_{\G}+Q_{\G})-\hat{h}_{E_{\G}}(P_{\G})-\hat{h}_{E_{\G}}(Q_{\G}) \right).$$
It is a standard result that these pairings are bilinear on the corresponding abelian groups. However, it is more subtle when it comes to positive definiteness on the free part of the abelian group. In general, $\< \cdot  , \cdot \>_E$ is positive definite on the quotient of $E(K)$ by the subgroup generated by $E(K)_\text{tors}$ and  $K/\CQ\text{-trace of }E(K)$ (\cite{LFD}, Chapter 6, Theorem 5.4). Since we assume that the $K/\CQ$-trace of $E$ is zero, it follows that $\< \cdot  , \cdot \>_E$ is a positive definite bilinear form on $E(K)/E(K)_\text{tors}$. Let $P^1, \ldots, P^r$ be a set of generators for the free part of $E(K)$. The positive definiteness of the canonical height pairing on $E(K)/E(K)_\text{tors}$ implies that
$$ \det \left(\<P^i, P^j\>_E\right)_{1\leq i,j\leq r} \neq 0.$$ Let $B_E$ be the union of all $B_{P^i}$'s as in Theorem A, then $B_E$ is a finite set of codimension-two subvarieties such that for any smooth hypersurface $\G \neq H_\infty$ that does not contain any subvariety in $B_E$ and is not a component of the discriminant locus of $E$ (so that $E_\G$ is non-singular), then the equality in Theorem A holds for $P^1, \ldots, P^r$. By multilinearity of determinant, we have
\begin{align*}   \frac{\det \left( \<P^i_\G, P^j_\G\>_{E_\G} \right)_{1\leq i,j\leq r}}{(\deg \G)^r} &=\det\left( \frac{\<P^i_\G, P^j_\G\>_{E_\G}}{\deg\G} \right)_{1\leq i,j\leq r} \\
&=\det \left(\<P^i, P^j\>_E\right)_{1\leq i,j\leq r}\\
&\neq 0.
\end{align*}

On the other hand, lemma \ref{LNS} implies that there exist $N_1, \ldots, N_r$ such that the locus for which the specialization of any one of $[N_1]P^1, \ldots, [N_r]P^r$ is a singular point is of dimension at most $n-2$. So this locus contains only a finite number of subvarieties of dimension $n-2$ and we let $S_E$ be the set of all these codimension-two subvarieties. If $\PP^{n-1}\cong\G \subset \PP^n$ is a hyperplane such that $E_\G$ is non-singular, we can view $E_\G$ as as an elliptic curve defined over $\CQ(\PP^{n-1})$. Furthermore, if $\G$ does not contain any of the subvarieties in $S_E$, then lemma \ref{LNS} holds for $P^1_\G, \ldots, P^r_\G$ without the assumption that the reduced Weierstrass equation of $E_\G$ is minimal.

To summarize, if $\G \neq H_\infty$ is a hyperplane not contained in the complement of $U$, does not contain subvarieties in $B_E \cup S_E$ and is not a component of the discriminant locus of $E$, then \begin{equation}\det \left( \<P^i_\G, P^j_\G\>_{E_\G} \right)_{1\leq i,j\leq r} \neq 0 \label{LIG}\end{equation} and lemma \ref{LNS} holds for $P^1_\G, \ldots, P^r_\G$. The set of such hyperplanes $\G \subset \PP^n$ is Zariski dense in the dual variety ${\PP^n}^*$ (the set of hyperplanes in $\PP^n$) because the conditions of not containing a finite number of subvarieties and not equaling any one of a certain finite number of hyperplanes are open conditions. Thus, there are infinitely many hyperplanes $\G$ that allow us to reduce the dimension from $\PP^n$ to $\PP^{n-1}$. As mentioned before, although we may not have the minimality of Weierstrass equation in the $\PP^{n-1}$ case, but we do know that lemma \ref{LNS} holds for $P^1_\G, \ldots, P^r_\G$, which is enough for Theorem A to hold true for $P^1_\G, \ldots, P^r_\G$ on $E_\G(\CQ(\G))$ and allow us to apply inductively this reduction on the dimension of the projective space and eventually reduce to the case $\PP^2$, i.e. we may assume $E$ is an elliptic curve defined over $K=\CQ(\PP^2)$ with $P^1, \ldots, P^r \in E(K)$ such that $\det \left( \<P^i, P^j\>_{E} \right)_{1\leq i,j\leq r} \neq 0$. 

We continue to reduce the dimension to one as above but do not restrict ourselves to just hyperplanes, i.e. we can choose any smooth curve $C \subset \PP^2$ of degree $d$ not contained in the complement of $U$, not a component of the discriminant locus of $E$ and does not contain subvarieties in $B_E$. The union of all such curves is the Zariski open set $\PP^2\backslash B_E$, because given any point $t \in \PP^2(\CQ)\backslash B_E$, we can always find a smooth curve $C/\CQ$ of degree $d$ such that it contains $t$ but not points in $B_E$ and it is not a component of the discriminant locus of $E$ and $\PP^2\backslash U$. With this construction, over these infinitely many curves $C$ of degree $d$, we have $\det \left( \<P^i_C, P^j_C\>_{E_C} \right)_{1\leq i,j\leq r} \neq 0$. By continuity of the determinant function and Silverman's theorem (\cite{SHS}, Theorem B), we conclude that 
$\det \left( \<{P^i_C}_t, {P^j_C}_t\>_{{E_C}_t} \right)_{1\leq i,j\leq r} \neq 0,$ for all $t \in C(\CQ)$ with large enough $h_C(t)$. This means that the specialized points ${P^1_C}_t, \ldots, {P^r_C}_t$ are $\ZZ$-linearly independent. Since we always choose the hyperplanes (or smooth curves in the last induction step) not being contained in the complement of $U$, the specialization maps for $t \in U(\CQ)$ commute with reduction maps as shown in Figure 1.
\begin{figure}
$$\xymatrix{
E(\CQ(\PP^n)) \ar@{->}[r]^{\s_t} \ar@{->}[d]_{\text{red mod $\G$}} 
&E_t(\CQ)  \\
E_\G((\CQ(\G))) \ar@{->}[d] & \\
\vdots \ar@{->}[d]_{\text{red mod $C$}} & \\
E_C(\CQ(C)) \ar@{->}[uuur]_{\s_t} 
}$$
\caption{}
\end{figure}
So for $t \in C(\CQ) \cap U(\CQ)$ with $h_C(t)$ large enough, we have injectivity of $\s_t$ on both the free part and the torsion points of $E(K)$. 

Lastly, we will prove by induction on $n$ that the union of all such curves in Theorem B is Zariski dense in $\PP^n$. We have just shown in the previous paragraph that the case $n=2$ is true. For $E/\CQ(\PP^n)$ in general with $n \geq 2$, we have shown that the set of hyperplanes $\G \subset \PP^n$ such that condition $(\ref{LIG})$ holds on $E_\G/\CQ(\G\cong \PP^{n-1})$ is Zariski dense in ${\PP^n}^*$, thus the union of all such $\G$ (call them good $\G$) is Zariski dense in $\PP^n$. By induction hypothesis, the union of curves satisfying Theorem B for $E_\G/\CQ(\G\cong \PP^{n-1})$ (call them good curves of degree $d$ for $\G$) is Zariski dense in $\G$. So the union 
$$ \bigcup_{\text{good $\G$}} \text{union of good curves of degree $d$ for $\G$}$$ is Zariski dense in $\PP^n$.
\end{proof}

\begin{acknowledgement}
I would like to thank my advisor, Joseph Silverman, for many enlightening discussions and encouragement. Also, many thanks to the referee for the valuable and insightful comments and suggestions
\end{acknowledgement}


\end{document}